\documentclass[a4paper,11pt]{article}

\usepackage{amsmath}
\usepackage{amsthm}
\usepackage{amssymb}
\usepackage{amsfonts}
\usepackage{hyperref}
\usepackage[all]{hypcap}[2006/02/20]

\newcommand{\mb}[1]{\mathbf{#1}}
\newcommand{\bb}[1]{\mathbb{#1}}
\newcommand{\Rp}[1]{\mathbb{R}^{#1}}

\newcommand{\pard}[2]{\frac{\partial #1}{\partial #2}}
\newcommand{\mbf}[1]{\mathbf{#1}}
\newcommand{\mbb}[1]{\mathbb{#1}}

\newcommand{\ho}{\left(\frac{d}{dt} -\Delta \right)}
\newcommand{\ddt}[1]{\frac{ d #1}{dt}}
\newcommand{\ip}[2]{\left \langle #1 , #2 \right\rangle}
\newcommand{\n}{\nabla}
\usepackage[pdftex]{graphicx}

\begin{document}
\theoremstyle{plain}
\newtheorem{theorem}{Theorem}[section]
\newtheorem{lemma}[theorem]{Lemma}
\newtheorem{claim}[theorem]{Claim}
\newtheorem{prop}[theorem]{Proposition}
\newtheorem{cor}[theorem]{Corollary}

\theoremstyle{definition}
\newtheorem{defses}[theorem]{Definition}
\newtheorem{assumption}[theorem]{Assumption}

\theoremstyle{remark}
\newtheorem{remark}[theorem]{Remark}
\hyphenation{Min-kow-ski}

\begin{center}
{\LARGE \bfseries{The perpendicular Neumann problem for mean curvature flow with a timelike cone boundary condition}}
\\[20pt]
by \\
\textsc{Ben Lambert} \\
\emph{Durham University}\\
\texttt{b.s.lambert@durham.ac.uk}
\end{center}

\begin{abstract}
This paper demonstrates existence for all time of mean curvature flow in Minkowski space with a perpendicular Neumann boundary condition, where the boundary manifold is a convex cone and the flowing manifold is initially spacelike. Using a blowdown argument, we show that under renormalisation this flow converges towards a homothetically expanding hyperbolic solution.
\end{abstract} 
\begin{center}
\end{center}
\section{Introduction}

In this paper we work with Mean Curvature Flow (MCF) of hypersurfaces in Minkowski space with a perpendicular Neumann boundary condition. We show that if the boundary manifold is a convex cone made up of timelike rays then \emph{any} initially spacelike hypersurface satisfying the boundary condition will exist for all time. If we renormalize to keep volume of the flowing manifold constant then we see that this hypersurface converges to a homothetic solution to MCF, specifically that given by flowing a hyperbolic hyperplane (see Definition \ref{homothetic}). Since these hyperbolic hyperplanes are set of points of constant negative ``radius'' in Minkowski space, this may be seen as a Minkowski--Neumann analogue of \cite{HuiskenConvex}.

MCF with a perpendicular Neumann condition in the Euclidean setting was considered as a graph over $C^{2+\alpha}$ domains by Huisken in \cite{Huiskengraph}, where a  longtime existence and convergence result was obtained. Altschuler and Wu \cite{AltschulerWu} also considered graphs with a varying angle Neumann condition although only with flowing surfaces of dimension 2 and over convex domains, yielding convergence to translating solutions. In \cite{Stahlsecond} and \cite{Stahlfirst} Stahl dealt with perpendicular boundary conditions but this time not necessarily as a graph. He showed that when the boundary condition was an umbilic manifold and the initial surface was convex, the flowing manifolds converged to round points.  Buckland \cite{Buckland} found boundary monotonicity formulae for MCF and classified Type I boundary singularities for $H>0$ with a perpendicular Neumann boundary condition. 

MCF (and related flows) have been used in Minkowski space by Ecker and Huisken in \cite{EckerHuiskenCMCMink} to construct entire surfaces of constant mean curvature. In \cite{EckerMinkowskiDBC} Ecker dealt with both the Dirichlet boundary condition and the entire case obtaining that without any growth conditions on the initial data, the entire flow exists for all time. He also showed that for any graphical spacelike initial data MCF with a Dirichlet boundary condition exists for all time and converges to a maximal surface. In \cite{EckerNull} he generalized this to the degenerate Dirichlet boundary case in more general semi-Riemannian manifolds. For a recent application of spacelike MCF see the work of Guilfoyle and Klingenberg \cite{GuilfoyleKlingenbergCaratheodory}. 

This paper is part of the author's PhD thesis. The author would like to thank his supervisor, W. Klingenberg, for his support, enthusiasm and excellent advice. 

\subsection{Definitions and notation}
We will need several definitions: Define $\bb{R}^{n+1}_1$ (where $n\geq2$) to be Minkowski space that is $\bb{R}^{n+1}$ equipped with the indefinite metric $\ip{-}{-}$ where \[\ip{\mb{x}}{\mb{y}}=x_1y_1+ \ldots +x_ny_n -x_{n+1}y_{n+1} \]
Let $\Sigma$ be a smooth embedded manifold in $\bb{R}^{n+1}_1$ with an indefinite metric, which is to say it is possible to find locally $n-1$ orthogonal vector fields $X_i$ and $Y$ such that $\ip{X_i}{X_i}=1$ and $\ip{Y}{Y}=-1$. This will be the boundary manifold, and let $\mu$ be its outward pointing unit normal. Our intention is to flow $M^n$, a smooth $n$--dimensional topological manifold with boundary $\partial M$, from an initial spacelike embedding $\mb{F}_0:M^n \rightarrow \mbb R^{n+1}_1$. We specify additionally that $\mbf{F}_0( \partial M) \subset \Sigma$ with the extra compatibility condition $\ip{\nu}{\mu}=0$, where $\nu$ is the normal to $M_0=\mbf{F}_0(M^n)$.
\begin{defses}
Let $\mathbf{F}: M^n \times [0,T] \rightarrow \bb{R}^{n+1}_1$ be such that
\begin{equation}
\label{pmcf}
\begin{cases}
\frac{d \mathbf{F}}{dt} = \mathbf{H}= H \nu & \forall (x,t) \in M^n \times [0,T]\\
\mathbf{F}(\cdot,0)= M_0&\\
\mathbf{F}(x,t) \subset \Sigma & \forall (x,t) \in \partial M^n \times [0,T]\\
\ip{\nu}{\mu}(x,t)=0 & \forall (x,t) \in \partial M^n \times [0,T]\\
\end{cases}
\end{equation}
then $\mathbf{F}$ moves by \emph{Mean Curvature Flow with a Neumann free boundary condition $\Sigma$} (here $\nu(x,t)$ is the normal to $\mbf{F}$ at time $t$.)
\end{defses} 

We will need various geometric quantities on various manifolds. A bar will imply quantities on $\mbb{R}^{n+1}_1$, for example $\overline \Delta, \overline \nabla, \ldots$ and so on; no extra markings $\Delta, \nabla, \ldots $ will be geometric quantities on $M_t$ our flowing surface at time $t$ and for any other manifold $Z$ $ \Delta^Z,  \nabla^Z, \ldots \text{etc.}$ will be the Laplacian, covariant derivatives, $\ldots$  on $Z$.

In this paper we will choose $\Sigma$ to be a timelike cone -- a cone in Minkowski space such that at any but the singular point the tangent space has a strictly timelike vector (see below for full details). 

In Minkowski space we have the equivalent of the homothetically shrinking sphere -- the homothetically expanding hyperbolic plane.
\begin{defses}
\label{homothetic}
We define the \emph{expanding hyperbolic hyperplane} $\mb{G}_k$ to be the solution to (\ref{pmcf}), starting with the section of hyperbolic plane of ``radius'' $k$ inside the cone $\Sigma$. That is at time $t=0$, $\ip{\mb G_k}{\mb G_k}=-k^2$ with $(\mb{G}_k)_{n+1}>0$. It is easy to show that the boundary conditions are satisfied and that
\[
 -\ip{\mbf G_k}{\mbf G_k} =k^2 + 2nt
\]
\end{defses}

In this paper we obtain firstly the following longtime existence result:
\begin{theorem}
 Let $\Sigma$ be a convex cone. Given that $M_0$ is initially spacelike then a solution to equation (\ref{pmcf}) exists for all time. Furthermore this solution stays between two homothetic solutions $\mb G_{C_0}$ and $\mb{G}_{C_1}$ where $C_0$ an $C_1$ are the minimum and maximum values of $\sqrt{-<\mb{F}, \mb{F}>}$ at time $t=0$.
\end{theorem}

We are able to get further convergence results. We define $\widehat{M}$ to be the blowdown of $M$, that is $M$ renormalised by dialations so that $\widehat{M}$ has constant unit area. By defining convergence ``at infinity'' to be the convegence of $\widehat{M}$, we get the following:

\begin{theorem}
 Any initially spacelike solution of equation (\ref{pmcf}) with a convex cone boundary condition under renormalisation will converge to a homothetic solution in the $C^1$ norm. Further, there exists an increasing sequence of $t_i$ such that $\widehat{M}_{t_i}$ converge to the solution on the interior in the $C^\infty$ topology. 
\end{theorem}

\section{A reparametrisation}

For simplicity we may reparametrise the above system as a graph over a topological disc $D \subset B^n_1(\mbf{0})$ defined by the intersection of the interior of $\Sigma$ with the hyperplane perpendicular to $\mbf{e}_{n+1}$ and intersecting $(0, \ldots, 0, 1)$. We may then describe a spacelike manifold $M$ inside $\Sigma$ as follows: At a point $\mb{x} \in D$ if we take the ray from $\mbf 0$ through $\mbf x$ then the ray will intersect $M$ only once. If $\mbf p$ is that point of intesection then let $u(\mbf x)= \sqrt{-\ip{\mbf p}{\mbf p}}$. The graph $u$ now parametrises $M$ by $\mb{F}(x)= u(x)\frac{\mb{x}+\mb{e}_{n+1}}{\sqrt{1- |\mb{x}|^2}}$. Standard calculations give geometric quantities, for example:
\[ g_{ij}= \frac{u^2}{1-|\mb{x}|^2} \left( \delta_{ij} + \frac{x_ix_j}{1-|\mb{x}|^2} \right) - D_iuD_ju\]
and
\[ \nu=\frac{({1-|\mb{x}|^2}) \mb D u + u \mb x + (Du \cdot x ({1-|\mb{x}|^2})+u)\mb{e}_{n+1}}{({1-|\mb{x}|^2})v}\]
where $v$ is a gradient-like function
\[
 v= \sqrt{ \frac{u^2}{1-|\mbf x|^2} +(Du. \mbf x)^2 - |Du|^2}\ \ .
\]

By this method we see that a solution to MCF is eqivalent to a solution to the following parabolic quasilinear PDE:
For $u:D\times [0,T) \rightarrow \bb{R}$ then
\begin{equation}
\label{gmac}
 \begin{cases}
  \ddt{u} = \frac{vH\sqrt{1- |\mb x|^2}}{u} =  g^{ij}D_{ij} u +\frac{n+1}{u} - \frac{1}{v^2} \left( \frac{u}{1-|\mbf x|^2} + 2 Du.\mbf x \right) & \forall \mb x \in D\\
u(\mbf x, 0)=u_0(\mbf{x})& \forall \mb x \in D\\
Du \cdot (\gamma-\gamma\cdot x \mb{x})=0& \forall \mb x \in \partial D
 \end{cases}
\end{equation}
where  
\begin{flalign*}
 g^{ij}= \frac{1-|\mbf x|^2}{u^2}\Bigg( \delta_{ij} + \frac{1}{v^2} \Bigg[ &\Bigg( |Du|^2 -\frac{u^2}{1-|\mbf x|^2} \Bigg) x_i x_j \\
& + D_i u D_j u - Du. \mbf x \left( x_i D_j u + x_j D_i u \right) \Bigg] \Bigg)
\end{flalign*}
is the inverse of the metric and $\gamma$ is the outward pointing unit normal to $D$. Long-term existence is equivalent to uniform parabolicity of the above equation and $C^1$ bounds on $u$. By calculating eigenvalues of the metric $g_{ij}$ we see that this is equivalent bounding $\max\left\{\frac{1}{v^2}, \frac{1}{u^2}, u^2 \right\}$  from above. 

\section{The boundary manifold}

Here we will define more rigorously the boundary manifold $\Sigma$ and state formulae for its curvature. Let 
$ \widetilde{\mb S}: S^n \rightarrow B_1(0) \subset \mbb R ^{n}$
 be a smooth embedding of a sphere into the unit ball centred at the origin with outward unit normal $\mb n$. Then we may define a boundary cone $\Sigma_{\widetilde{\mb S}}$ (later the subscript will be dropped) by embedding $\mbb{R}^n$ into $\mbb{R}^{n+1}_1$ at height $1$ and then defining $\Sigma_{\widetilde{\mb S}}$ to be the set of all rays going through the origin and some point $(\widetilde{\mb{S}}(x), 1)$. More explicitly we may give a parametrisation $\mb S : (0, \infty) \times S^n \rightarrow \mbb{R}^{n+1}_1$ of $\Sigma_{\widetilde{\mb S}}$ by 
\[ (l, x) \mapsto l \widetilde{\mb S}(x) + l \mb{e}_{n+1}\]
We may now calculate all quantities needed. For example, we may now see that in these coordinates:
\[ A^\Sigma\left(\cdot ,\pard{}{l}\right) =0\]
\[ A^\Sigma \left(\pard{}{x^i}, \pard{}{x^j}\right)= \frac{l A^{\widetilde{S}} \left(\pard{}{x^i}, \pard{}{x^j}\right)}{\sqrt{1-\ip{\widetilde{\mb{S}}}{ \mb n}^2}}\]
Therefore for a orthonormal set vectors $\mb{e}_l,\mb{e}_1, \ldots, \mb{e}_{n-1}\in T_{\mb{p}} \Sigma$ obtained by picking orthonormal coordinates on $\widetilde S$ and renormalising then
\[A^\Sigma(\mb{e}_i, \mb{e}_j)=\frac{A^{\widetilde{S}}(\mb{e}_i, \mb{e}_j)}{l \sqrt{1-\ip{\mb{\widetilde{S}}}{ \mb{n}}^2}}\ \ .\]
Hence we can see that, as we would expect, convexity of $\Sigma$ is equivalent to convexity of the embedding $\widetilde{\mb{S}}$, the second fundamental form has a zero eigenvector along the timelike rays from the origin, and the second fundamental form decreases linearly as you move up the cone. 

\section{Evolution equations}
In this section we will derive some useful evolution equations by straightforward calculation.
We define the following:
\[ F^2 = -\ip{\mbf{F}}{\mbf{F}} \  >0 \]
\[ S=\ip{\mbf{F}}{\nu} \]
We may think of these as in some sense $C^0$ and $C^1$ measures of how far our flowing manifold is from a homotheic solution $\mb{G}_k$.
\begin{lemma}
\label{evolF2}
Under MCF we have
\[
 \ho F^2 = 2n 
\]
\end{lemma}
\begin{proof}
We have
\[ 
 \frac{ d F^2}{dt}= -2 H \ip{\nu}{\mbf{F}}= 2HS \ \ ,
\]
and 
\begin{flalign*}
  \Delta F^2 &=-2 g^{ij}\left( h_{ij}\ip{\nu}{\mbf{F}} + g_{ij} \right)\\
&=2HS-2n\ \ \ .
\end{flalign*}
\end{proof}
\begin{lemma}
\label{dnudt}
 \[
\frac{ d \nu}{ dt} = \nabla H
 \]
\end{lemma}
\begin{proof}
See for example \cite[Proposition 3.1]{EckerHuiskenCMCMink} .
\end{proof}
\begin{lemma}
On the interior of the flowing manifold we have
\label{evolS}
 \[
  \ho S = 2H - S | A |^2
 \]
\end{lemma}
\begin{proof}
 Using Lemma \ref{dnudt} we get
\begin{flalign*}
 \frac{d S}{d t} = -H \ip{ \nu}{\nu} - \ip{\mbf{F}^\top}{ \nabla H}\\
&= H - \ip{ \mbf{F}^\top}{ \nabla H}
\end{flalign*}
and
\begin{flalign*}
 \Delta S &=-g^{ij} \left( \nabla_\pard{}{x^i} A(\mbf F^\top, \pard{}{x^j}) + A(\nabla_\pard{}{x^i} \mbf F^\top,  \pard{}{x^j}) \right)\\
&=-\nabla_{\mbf F ^\top} H - g^{ij}  A(\nabla_\pard{}{x^i} \mbf F^\top,  \pard{}{x^j}) \ \ .
\end{flalign*}
Now we calculate
\begin{flalign*}
 g^{ij}A\left(\nabla_\pard{}{x^i} \mbf F^\top,  \pard{}{x^j}\right)&=g^{ij}A\left(\left(\overline\nabla_\pard{}{x^i} (\mbf F - S \nu)\right)^\top,  \pard{}{x^j}\right)\\
&=H -S |A|^2   
\end{flalign*}
which shows
\[
 \Delta S =- \nabla_{\mbf F ^\top} H - H +S|A|^2 \ \ .
\]
Hence the Lemma.
\end{proof}
We also need several evolution equations for the curvature.
\begin{prop}
 \label{evolcurv}
On the interior of the flowing manifold we have for $m\geq1$
\begin{flalign*}
\ho &H =- H |A|^2\\
 \ho &|A|^2 = - 2 |A|^4 -2 |\nabla A|^2\\
\ho& |\nabla^m A|^2 \leq -2 |\nabla^{m+1}A|^2\\
&\qquad\qquad\qquad +\underset{i+j+k=m}\sum \nabla^i A *\nabla^j A*\nabla^k A*\nabla^m A
\end{flalign*}
\end{prop}
\begin{proof}
 See \cite[Proposition 3.3]{EckerHuiskenCMCMink}.
\end{proof}

\section{Boundary derivatives}
\label{sectionboundary}
To apply Hopf maximum principle we also need to consider derivatives of functions at the boundary in the direction of $\mu$, the normal to $\Sigma$. As in the case of Stahl \cite{Stahlsecond} these identities come from derivatives of the boundary condition. We first demonstrate the following simple result.	
\begin{lemma}
\label{bdryF2}
 For $\mbf p \in \partial M^n \times [0, T)$ we have
\[
 \ip{\nabla F^2}{ \mu } =0
\]
\end{lemma}
 \begin{proof}
 We have that $\nabla F^2 = (\overline \nabla F^2)^\top$. Furthermore we have that $\overline \nabla F^2 \in T_{\mbf p} \Sigma$ and hence we have
\begin{flalign*}
 \ip{\mu}{ \nabla F^2}&=\ip{\mu}{ \overline \nabla F^2+\ip{\nu}{ \nabla F^2}\nu} =0\ \ .
\end{flalign*}
 \end{proof}
Now we take spatial derivatives of the boundary condition to give:
\begin{lemma}
\label{bdry2nd} 
For $W \in T_p M_t \cap T_p \Sigma$ then
\[ A(\mu, W)= -A^\Sigma(\nu, W)\]
\end{lemma}
\begin{proof}
 \begin{flalign*}
  0&=W\ip{\nu}{ \mu}_{\Rp{n+1}_1} =\ip{\nabla_W \nu}{ \mu}+\ip{\nu}{ \nabla^\Sigma_W \mu}_\Sigma =A(W, \mu)+A^\Sigma(\nu,W) \ \ .
 \end{flalign*}
\end{proof}
For our gradient estimate we also need
\begin{lemma}
\label{bdryS}
 For $p \in \partial M^n \times [0,T)$ we have
\[\ip{\nabla S}{\mu}=-S A^\Sigma(\nu, \nu)\]
\end{lemma}
\begin{proof}
\begin{flalign*}
 \nabla S &=-\ip{\mbf{F}}{ \pard{\nu}{x^i} }g^{ij}\pard{}{x^j} =- A(\mb{F}^\top, \pard{}{x^i}) g^{ij}\pard{}{x^j}
\end{flalign*}
and so taking an inner product and applying Lemma \ref{bdry2nd}
\begin{flalign*}
 \ip{\nabla S}{ \mu} = -A(F^\top, \mu) = A^\Sigma(F^\top, \nu)\ \ .
\end{flalign*}
Using the fact that the second fundamental form of $\Sigma$ has a zero eigenvector in the direction $\mb{F}$ then
\begin{flalign*}
 A^\Sigma(F^\top, \nu) =A^\Sigma(\mb{F}, \nu) - S A^\Sigma(\nu, \nu)=-SA^\Sigma(\nu, \nu)\ \ .
\end{flalign*}
\end{proof}

Now differentiating the boundary condition with respect to time:
\begin{lemma}
\label{bdryH}
 For $p \in \partial M^n \times [0,T)$ we have
\[\ip{\nabla H}{ \mu}=-HA^\Sigma(\nu, \nu)\]
\end{lemma}
\begin{proof}
Using Lemma \ref{dnudt} 
 \begin{flalign*}
  0&=\ddt{}\ip{\nu}{ \mu_{|_{\mbf F}} }_{\Rp{n+1}_1}=\ip{\nabla H}{ \mu} + \ip{\nu}{ D\mu( H \nu )}_{\Rp{n+1}_1}=\ip{\nabla H}{ \mu} + H A^\Sigma(\nu, \nu)\ .
 \end{flalign*}
\end{proof}
\begin{remark}
 We note that if $\Sigma$ is convex then the normal derivatives at the boundary of both $H$ and $S$ are negative. 
\end{remark}
On the other hand regardless of the boundary we are able to get
\begin{cor}
\label{bdryHS}
 For $p \in \partial M^n \times [0,T)$ we have
\[\ip{\nabla \frac{H}{S}}{ \mu}=0\ \ .\]
\end{cor}
 
\section{Gradient estimate}
We now obtain a gradient estimate, that is to say, a lower bound on $v$. Note that in the graphical notation of equation (\ref{gmac}) 
\[
 S=\frac{u^2}{v\sqrt{1- |\mbf x|^2}}\ \ .
\]
Hence it is sufficient to find a suitable upper bound on $S$ and a lower bound on $u^2=F^2$. We will need an assumption:
\begin{assumption}
 We will assume from here on that $\Sigma$ is convex . 
\end{assumption}
 And also a maximum principle:
\begin{theorem}[Weak Maximum Principle]
\label{WMP}
 Suppose we have a function $f:M^n \times [0,T) \rightarrow \mbb{R}$ then if $f$ satisfies
\begin{equation*}
 \begin{cases}
  \displaystyle \ho f(\mbf p , t) \leq 0 & \forall (\mbf p , t) \in M^n \times [0,T) \text{ such that } \nabla f(\mbf p) =0\\
    \ip{\nabla f }{ \mu } \leq 0 & \forall (\mbf p , t) \in \partial M^n \times [0,T)\\
 \end{cases}
\end{equation*}
then $f(\mbf x ,t ) \leq \underset{\mbf p \in M^n} \sup f ( \mbf p ,0 )$ for all $(\mbf x , t) \in M^n \times [0,T)$.
\end{theorem}

Using Lemmas \ref{evolF2} and \ref{bdryF2}, we see we can immediately apply the above to both $F^2-2nt$ and $2nt-F^2$ to give
\begin{equation}
 C_1(M_0) \leq F^2-2nt \leq C_2(M_0)
\label{F2estimate}
\end{equation}
This may be interpreted as if our manifold lies between two copies of a hyperbolic soliton $\mb{G}_{C_1}$ and $\mb{G}_{C_2}$ initially, then it will do so for all time. It also gives the required bounds on $u$, and further ensures that $M_0$ stays away from the singularity of $\Sigma$ for all the time a solution exists. 

Now using Proposition \ref{evolcurv} and Lemma \ref{evolS} we consider the evolution of $\frac{H}{S}(C+2nt)$ we see
\[ \ho\frac{H}{S}(C+2nt) =\frac{H}{S}\left(2n- 2(C+2nt)\frac{H}{S} \right) - 2\ip{\frac{\nabla S}{S}}{\nabla \frac{H}{S}(C+2nt)}\]
where the final term will disappear at a stationary point. Hence given that $H>0$ on $M_0$ and again applying weak maximum principle we have for $C_3, C_4>0$
 \[\frac{C_3}{C+2nt} \leq \frac{H}{S} \leq \frac{C_4}{C+2nt}\ \ , \]
or for $\widehat{C}_3, \widehat{C}_4>0$
\begin{equation}
\label{HSineq}
\widehat{C}_3 \leq \frac{H}{S}F^2 \leq \widehat{C}_4 \ \ .
\end{equation}
If $H\geq 0$ on $M_0$ (with equality somewhere) then the constant $C_3$ is zero. This estimate implies preservation of weak or strict mean convexity since
\[
 H\geq C_3 \frac{S}{C+2nt} \geq 0\ \ .
\]
If we neglect the assumption of initial mean convexity, estimate (\ref{HSineq}) still holds, although  $\widehat{C}_3 \leq -n$.

Until now we have not used our assumption, and this is the point at which it comes in, in the form of a sign of the boundary derivative on $H$ (and later $S$). Using Proposition \ref{evolcurv} we get that on the interior of $M$
\[\ho H^2 = -2H^2|A|^2 - 2 |\nabla H|^2\ \ ,\]
 and from Lemma \ref{bdryH} and our assumption, $\nabla_\mu H =-2H^2 A(\nu,\nu)\leq0$. By the Weak Maximum Principle we therefore have 
\[ H^2 < C_5 \ \ .\]

Now using Lemmas \ref{evolF2}, \ref{evolS}, \ref{bdryF2}, \ref{bdryS} we calculate for $f=S- \frac{\sqrt{C_5}}{n}F^2$ that
\[
 \ho f = 2H -2\sqrt{C_5} -S|A|^2 \leq -S|A|^2 \leq 0 
\]
and
\[
 \ip{\nabla f}{ \mu} = \ip{\nabla S}{\mu} \leq 0 \ \ .
\]
Again applying the Weak Maximum Principle we see
\[
 S \leq C_6 (M_0) + \frac{\sqrt{C_5}}{n}F^2
\]
and hence we get
\[
 v > \frac{F^2}{\sqrt{1-|\mbf x|^2}\left(C_6 +\frac{\sqrt{C_5}}{n}F^2\right)}>0\ \ .
\]
We have the estimates required, and give the following summary:

\begin{theorem}
\label{LTE}
 Given that $M_0$ is spacelike, a solution to equation (\ref{pmcf}) exists for all time. Mean convexity is preserved by the flow and if the solution is initially bounded by $\mb{G}_{C_1}$ and $\mb{G}_{C_2}$ will remain so for all time.
\end{theorem}
\begin{proof}
 From the gradient estimate above and from our bound $\sqrt{C_1+2nt}\leq u\leq\sqrt{C_2+2nt}$ we have uniform parabolicity of equation (\ref{gmac}) and a bound on $|Du|$. Therefore for all time intervals $[0,T]$ from standard results on quasilinear PDE's, for example \cite{Lieberman}, we have existence of a unique smooth solution. Therefore we have existence of a solution to equation (\ref{pmcf}) for all time. 
\end{proof}

\section{Improvements to estimates}
We expect our solution to move towards an expanding hyperbolic hyperplane $\mb{G}_k$, but if this is so the estimates from the previous section are not optimal. We currently have that $F \leq S \leq C_6 + C_7 F^2$ while on a special solution $S=F$. Also we only have $\frac{C_3}{\sqrt{C+2nt}} \leq H \leq \sqrt{C_5}$ while on a special solution we know $H= \frac{n}{F}$. However our ratio of $H$ to $S$, estimate (\ref{HSineq}), is of the right order. 

To improve our estimates we consider
\[ \frac{| \nabla F^2 |^2}{F^2} = 4 \frac{| F^\top |^2}{F^2} = \frac{4}{F^2}\ip{F - S \nu}{ F - S \nu} = 4\frac{S^2 -F^2}{F^2}\ \ .\]
Note that this quantity is scaling invariant and is zero on our special solution. We show that this will in fact asypmtote to zero. We calculate for $J=\frac{S^2-F^2}{F^2}$ that
\begin{flalign*}
 \ho J &= \frac{1}{F^2}\ho (S^2-F^2) +\frac{2}{F^4}\ip{\n F^2}{\n(S^2 - F^2)}\\
&\qquad\qquad\qquad\qquad+(S^2-F^2)\left( -\frac{1}{F^4}\ho F^2- 2 \frac{|\n F^2|^2}{F^6} \right) \\
&=\frac{1}{F^2}\Big[ 4HS -2 S^2|A|^2 -2|\n S|^2 - 2n\\
&\qquad\qquad\qquad\qquad -2nJ -8J^2 +\frac{2}{F^2}\ip{\n F^2}{\n (S^2-F^2)}  \Big]\ \ .
\end{flalign*}
Since $|A|^2\geq \frac{H^2}{n^2}$ we estimate 
\begin{flalign*}
 4SH - 2S^2|A|^2 \leq 2n^2 -\frac{2}{n^2}\left(SH-n^2\right)^2\leq 2n^2\ \ .
\end{flalign*}
By Cauchy--Schwarz and Youngs inequalities we also see
\begin{flalign*}
\frac{1}{F^2}\ip{\n F^2}{\n(S^2 -F^2)}-|\n S|^2&\leq 2\frac{S}{F}\frac{|\n F^2|}{F}|\n S|-\frac{|\n F^2|^2}{F^2}  -|\n S|^2\\
&\leq \frac{S^2}{F^2}\frac{|\n F^2|^2}{F^2}-\frac{|\n F^2|^2}{F^2}\\
&=4J^2\ \ .
\end{flalign*}
Applying to the evolution equation for $J$ we have
\begin{flalign*}
 \ho J&\leq\frac{2n}{F^2}\left[ n -1 -J\right]
\end{flalign*}
which implies since the boundary derivative of J is $\ip{\n J}{\mu} \leq 2\frac{S}{F^2}\ip{\n S}{\mu}<0$ that $J$ is bounded by the maximum of its initial value and $n-1$. But we can do better than that. For $C\geq C_2$ we see
\begin{flalign*}
\ho &J \log(C +2nt) \\
&\leq \frac{2n\log(C+2nt)}{F^2}\left[n-1-J\right] +\frac{2nJ}{C+2nt}\\
&\leq \frac{2n\log(C+2nt)}{F^2}\left[n-1-\left(1-\frac{1}{\log(C+2nt)}\right)J\right]
\end{flalign*}
 and by choosing $C$ sufficiently large, for example $C>e^2$ then we have the following:
\begin{prop}
\label{betterS2}
There exists constants $C_S,D_S, \widetilde{D}_S>0$ depending only on $n$ and $M_0$ such that
\[\frac{|\n F^2|^2}{F^2}\leq 4\frac{C_S}{\log(D_S+2nt)}\leq 4\frac{C_S}{\log(\widetilde{D}_S+F^2)}\]  
or equivalently
\[\frac{S^2}{F^2}\leq 1+\frac{C_S}{\log(D_S+2nt)}\leq 1+\frac{C_S}{\log(\widetilde{D}_S+F^2)}\]
\end{prop}
 Now the estimate equation (\ref{HSineq}) implies the following:
\begin{cor}
There exist constants $C^H_1$ and $C^H_2>0$ such that
\label{betterH}
 \[ C^H_1 \leq HF \leq C^H_2 \]
Where $C^H_1$ is positive if $M_0$ is initially mean convex.
\end{cor}
\begin{remark}
 If we suppose $M_0$ is mean convex, we may use the above to calculate the \emph{average} of $H$ asymptotically in time. From the proof of Lemma \ref{evolF2} we have that $\Delta F^2 = 2HS - 2n$ and therefore we see using Lemma \ref{bdryF2} and Divergence Theorem that
\[\int_M HS d\mu=n \int_M  d\mu \ \ .\]
Therefore since $H>0$  we may estimate
\[n\leq \frac{\int_M HF d\mu}{\int_M d\mu} \leq n\sqrt{1+\frac{C_S}{\log(D_S+2nt)}}\]
which asymptotes to exactly what we would expect on our special solution.
\end{remark}

\section{Interior curvature estimates}

We obtain some interior estimates on $|A|^2$ and its derivatives. Note that on our homothetic solution we get that $|A|^2=\frac{n}{F^2}$ and we search for estimates of a similar order. 

To construct a cutoff function first suppose $\overline K : \mbb{R}^{n+1}_1 \rightarrow \mbb{R}$ and define $K: M^n \times [0,T) \rightarrow \mbb R $ by $K(x,t)= \overline K (\mb F(x,t))$.

\begin{lemma}
 \[\ho K = - \overline \nabla_\nu \overline \nabla_\nu \overline K \big|_{\mb F} - \overline \Delta \: \overline K\big|_{\mb F} \]
\end{lemma}
\begin{proof}
 First
\[ \frac{d}{dt} K = H \overline \nabla_\nu \overline K \]
We also calculate
\begin{flalign*}
 \Delta K = H \overline \nabla_\nu \overline K + g^{ij}\overline \nabla_\pard{ \mb F}{x^i}\overline \nabla_\pard{ \mb F}{x^j} \overline K
\end{flalign*}
using the Weingarten relations. But now by considering locally in a suitable orthonormal coordinate system and noting the sign of $\nu$ we see
\[ \Delta K =  H \overline \nabla_\nu \overline K + \overline \Delta \: \overline K + \overline \nabla_\nu \overline \nabla_\nu \overline K\]
which gives the Lemma.
\end{proof}
We now stipulate an additional condition on $\overline K$, namely that $\overline \nabla_{\mb F} \overline K=0$. That is, the cutoff function is defined on a hyperbolic plane and remains constant on rays from the origin. We define $Y_\lambda=\left\{ \mb{x} |\ip{\mb{x}}{ \mb{x}}=-\lambda^2, x_{n+1}>0 \right\}$ that is a spacelike embedding of the hyperbolic plane of ``radius'' $\lambda$.
\begin{cor}
\label{evolK}
 Under the condition $\overline \nabla_{\mb F} \overline K=0$ we have at $\mb{p}\in M$
\begin{flalign*}
\ho K &\leq |{\nabla^{Y_F}}^2 \overline K |^{Y_F}(\mb{p}) \left(\frac{S^2}{F^2} - 1\right) - \Delta^{Y_F} \overline K (\mb{p})\\
&= |{\nabla^{Y_1}}^2 \overline K  |^{Y_1}\left(\frac{\mb{p}}{F}\right) \frac{\frac{S^2}{F^2} - 1}{F^2} - \frac{1}{F^2} \Delta^{Y_1} \overline K\left(\frac{\mb{p}}{F}\right)\\
&\leq \frac{\widetilde{C}_K}{F^2}
\end{flalign*}
\end{cor}
\begin{proof}
 Since $Y_F$ is perpendicular to $\mb F$ and $\overline \Delta \: \overline K$ has no contibution from the $\mb F$ direction we immediately have $\overline \Delta \: \overline K = \Delta^{Y_F} \overline K$. Similarly we have  
\begin{flalign*}
|\overline \nabla_\nu \overline \nabla_\nu \overline K \big|_{\mb F}| &= |\overline \nabla_{\nu- \frac{\ip{\nu}{ \mb F}}{F^2}\mb F} \overline \nabla_{\nu- \frac{\ip{\nu}{ \mb F}}{F^2}\mb F} \overline K \big|_{\mb F}|\\
&= |\overline \nabla^{Y_F}_{\nu+ \frac{S}{F^2}\mb F} \overline \nabla^{Y_F}_{\nu+ \frac{S}{F^2}\mb F} \overline K \big|_{\mb F}|\\
&\leq |{\nabla^{Y_F}}^2 \overline K|^{Y_F}\left(\frac{S^2}{F^2}-1\right)
\end{flalign*}
by Cauchy--Schwarz, giving the first inequality. 

The second is using the scaling of $K$ on $Y_F$ -- this allows us to estimate over $Y_1$ rather that $Y_F$ where $F$ may vary from point to point. This inequality is simply from properties of dilations and the constancy of $\overline{K}$ on rays from $\mb 0$: Keeping a function constant but dilating the manifold by $\lambda$ while keeping $\overline K$ the same we get that $g_{ij}$ becomes $\lambda^2 g_{ij}$, $g^{ij}$ becomes $\lambda^{-2} g^{ij}$, $\Gamma^k_{ij}$ remains $\Gamma^k_{ij}$ and so on. This gives the stated formula.

The third of these comes from estimating second derivatives of $\overline K$ on $Y_1$ and Proposition \ref{betterS2}.
\end{proof}

The final two conditions we wish $\overline{K}$ to have in addition to that of the above Corollary are:
\begin{itemize}
 \item $0\leq \overline{K} \leq C$ where $\overline{K}$ restricted to $Y_1$ has compact support and
\item $\frac{|\nabla^{Y_1} K|^2}{K}\leq \widehat{C}_K$ for some $\widehat{C}_K>0$. 
\end{itemize}

The question of whether such a function exists is easily solved, for example if we take the Poincar\'e model of hyperbolic space (which is  isometric to $Y_1$) we could take $K$ to be the radial function $K(r)=(1-Er^2)_+^3$. We then calculate in this metric
\[\frac{|\nabla^\text{Poin} K |^2}{K}=\frac{\left(\frac{dK}{dr}\right)^2 (1-r^2)^2}{4K} =  9E^2r^2(1-Er^2)_+(1-Er^2)\]
which is clearly bounded (depending on $E$) on the unit ball. Furthermore this function is zero outside a hyperbolic ball and bounded by $1$ and by changing $E$ we may choose the radius of the hyperbolic ball which is $\text{supp}(K)$.

For $K$ satisfying the above we know
\begin{flalign*}
|\nabla K|^2&= |\overline \n \, \overline{K}|^2+\ip{ \nabla^{Y_F} K}{ \nu-\frac{S}{F^2}\mb{F}}^2\\
&\leq \frac{\widetilde{C}_S |\nabla^{Y_1} K|^2}{F^2}
\end{flalign*}
 for $\widetilde{C}_S=1+\frac{C_S}{\log{D_S}}>0$ where we used the Cauchy--Schwarz inequality on the hyperbolic plane and Proposition \ref{betterS2}.  Therefore for such a function we have that at a maximum of $fK$ -- a point where $f \nabla K + K \nabla f=0$ then
\begin{flalign}
 \ho fK &\leq \frac{f\widetilde{C}_K}{F^2} +\widetilde{C}_Sf\frac{|\n^{Y_1} K|^2}{F^2 K} + K \ho f\nonumber\\
&\leq\frac{f C_K}{F^2} + K \ho f \label{evolcutoff}
\end{flalign}
where $C_K>0$.
\begin{lemma}
\label{A2interior}
 Let $L\subset \mbb R^{n+1}_1$ be such that if $\mb x \in L$ then $\mb \lambda x \in L \ \ \forall \lambda \in \mbb R$, and so that $Y_1 \cap L$ is a compact set of minimum hyperbolic distance $d>0$ from $\Sigma$ with a smooth boundary. Then on $M_t \cap L$
\[ |A|^2 \leq \frac{C_A}{F^2}\]
Where the constant depends on $d$, the second derivatives of the boundary of  $Y_1 \cap L$, $n$ and $M_0$.
\end{lemma}
\begin{proof}
 Since we have suitable cutoff functions (if necessary just using the radial one with sufficiently large $E$) then the proof comes down to evolution equations. We calculate that for $f_0=|A|^2(D+2nt)$:
\begin{flalign*} 
\ho f_0 &= 2n |A|^2 - 2(D+2nt)(|A|^4 + |\nabla A|^2)\\
&\leq  \frac{2nf_0}{D+2nt} - (D+2nt)|A|^4
\end{flalign*}
Now applying equation (\ref{evolcutoff}) we have by choosing $D$ large enough
\begin{flalign*}
 \ho Kf_0 \leq \frac{f_0}{F^2}\left[C_K +2nK  - B f_0K \right]
\end{flalign*}
for some $B>0$ depending on $C_1$ and $C_2$ (see equation (\ref{F2estimate})). Therefore since $K=0$ at the boundary we have the Lemma.
\end{proof}
\begin{lemma}
\label{nablaA2interior}
 For $L$ as in the previous Lemma we have that all \mbox{$m\geq1$} there exists a constant $C_{A,m}$ depending on $m, n, L, d, M_0$ and the second derivatives of the boundary of $Y_1\cap L$ such that
\[|\nabla^m A|^2 \leq \frac{C_{A,m}}{F^2}\]
\end{lemma}
\begin{proof}
 The proof is by induction. Writing $J_1=|A|^2(D+2nt) + E<C_A+E$ where $E>0$ is a constant yet to be chosen, we define
\[f_1=(D+2nt)|\nabla A|^2J_1\ \ .\]
Using Proposition \ref{evolcurv}, the Cauchy--Schwarz inequality and the above Lemma, writing $C_n$ for any positive constant depending only on $n$, $K$ and $M_0$ then
\begin{flalign*}
 \ho f_1 &\leq (D+2nt)\bigg[ -2J_1|\n^2A|^2 - 2(D+2nt)|\n A|^4\\
&\qquad\qquad\qquad+ \frac{C_n(E+1)}{D+2nt}|\n A|^2 +8(D+2nt)|\n^2 A||\n A|^2|A|\bigg]\\
&\leq(D+2nt)\bigg[\frac{4|\n A|^4|A|^2(D+2nt)^2}{(|A|^2(D+2nt) + E)^2}- 2(D+2nt)|\n A|^4\\
&\qquad\qquad\qquad\qquad\qquad\qquad\qquad\qquad\qquad+ \frac{C_n(E+1)}{D+2nt}|\n A|^2 \bigg]\\
&\leq(D+2nt)\bigg[\frac{C_n}{E^2}|\n A|^4(D+2nt)- 2(D+2nt)|\n A|^4\\
&\qquad\qquad\qquad\qquad\qquad\qquad\qquad\qquad\qquad+ \frac{C_n(E+1)}{D+2nt}|\n A|^2 \bigg]
\end{flalign*}
We now choose $E$ sufficiently large that the coefficient $\frac{C_n}{E^2}\leq\frac{1}{2}$ and therefore
\begin{flalign*}
 \ho f_1 &\leq (D+2nt)\left[ - \frac{3}{2}(D+2nt)|\n A|^4+ \frac{C_n}{D+2nt}|\n A|^2 \right]\\
&\leq \frac{C_n f_1}{D+2nt}- C_n f_1^2
\end{flalign*}
where here again we used the bound on $J_1$. Substituting into equation (\ref{evolcutoff}) we see:
\[\ho f_1K\leq f_1\left(\frac{C_n}{F^2} - C_n f_1\right)\]

In fact the same argument holds for
\[f_m=(B+2nt)|\nabla^m A|^2(|\nabla^{m-1} A|^2(B+2nt)+D)\]
by induction. This follows from:
\begin{itemize}
\item The fact that we don't use the extra negative term, $-|A|^4$, from the evolution equation of $|A|^2$.
\item The extra positive term from the evolution equation of $|\nabla^{m-1} A|^2$ is
\[\underset{i+j+k=m-1}\sum \nabla^i A *\nabla^j A*\nabla^k A*\nabla^{m-1} A \leq \frac{C}{(B+2nt)^2}\]
by the inductive hypothesis, which is of the right order for the above equations to still hold. 
\item The only other remaining difference is an extra term in the estimate on the evolution equation of $|\n^m A|^2$ which becomes a positive term $C_n \frac{\sqrt{f_m}}{D+2nt}$ in the evolution of $f_m$. This is easily delt with using $-\epsilon f_m^2$. 
\end{itemize}
Hence an identical argument to the above gives the Lemma.
\end{proof}

\section{Convergence and renormalisation}

In this section our purpose is to define the shape of the solution as $t \rightarrow \infty$. For this some notion of blowdown will be needed. 
\begin{defses}
 If $\mb{F}: M^n \times [0, \infty) \rightarrow \mbb{R}^{n+1}_1$ satisfies equation (\ref{pmcf}) then let $\widehat{\mb{F}} = \psi (t) \mb{F} $ where $\psi (t)$ is some factor such that the area of $\widehat{\mb{F}}(M)$ is 1. For any geometric quantity $f$ on $\mb F$ we will denote the same quantity $\widehat{f}$ on $ \widehat{ \mb F}. $

For $\mb{G}: M^n \times [0, \infty) \rightarrow \mbb{R}^{n+1}_1$ we will say $\mb{F} \rightarrow \mb{G}$ as $t \rightarrow \infty$  in $C^0$, $C^1$, $\ldots$ if $\widehat{\mb{F}} \rightarrow \widehat{\mb{G}}$ as $t \rightarrow \infty$ in $C^0$, $C^1$, $\ldots$ . 
\end{defses}
\begin{remark}
 It is usual (as in \cite{HuiskenConvex}) to renormalise time aswell. Indeed we may also do so here, by defining $s=\int_0^t \psi(r)^2 dr$. We then obtain
\[\frac{d \widehat{\mb{F}}}{ds}= \frac{d \widehat {\mb{F}}}{dt}\frac{d t}{ds} = \psi^{-2} \left( \psi H \nu - \psi^{-1}\frac{1}{n} \frac{\int_M H^2 d \mu}{\int_M d \mu} \mb{F}\right) = \widehat{H} \nu + \frac{\widehat{\mb{F}}}{n} \int_{\widehat{M}} \widehat{H}^2 d \widehat{\mu}\ \ .\]
In actual fact Lemma \ref{volest} will show that $s\geq C \log(D+2nt)$ and hence we need not make a distinction between $s \rightarrow \infty$ and $t\rightarrow \infty$. 
\end{remark}

It will be useful to have estimates on the quantity $\psi$:

\begin{lemma}
\label{volest}
 There exist constants $C_Y,\widetilde{C}_Y>0$ such that for time $t$ sufficiently large then
\[ \frac{C_Y}{F} \leq \psi (t) \leq C_Y\frac{\sqrt{1+\frac{C_S}{\log(\widetilde{D}_S +F^2)}}}{F}\leq \frac{\widetilde{C}_Y}{F}\ \ .\]
\end{lemma}

\begin{proof}
 Let $\mb{Y}$ be a parametrisation of $Y_1$. Then any spacelike manifold  contained within the lightcone may be written as $\mb{Z} = u(x) \mb Y (x)$. By standard methods we get the following identities:
\begin{flalign*} 
g^Z_{ij} &= u^2 g^Y_{ij} - D_iu D_ju\\
g_Z^{ij} &= \frac{1}{u^2}\left( g^{jk}_Y +\frac{D_aug_Y^{aj}D_bug_Y^{bk}}{u^2 - |\n^Yu|^2}\right)\\
\det (g_{ij}^Z) &= (u^2)^{n-1}(u^2-|\n^Y u|^2)\det g_{ij}^Y\ \ .
\end{flalign*}
From these we may see that
\[ |\nabla ^Y u |^2 = \frac{u^2|\n u^2|^2}{4u^2 +|\n u^2 |^2}\ \ .\]

Since $u= F$ and we may write the area of $M$ as the integral over the interior of $\Sigma$ intersected with $Y_1$ -- we call this set $\mathcal{B}$. Therefore:
\begin{flalign*}
\int_M d\mu &= \int_\mathcal{B} u^n\sqrt{1-\frac{|\n^Y u|^2}{u^2}} d \mu_{Y_1} \\
&= \int_\mathcal{B} F^n \sqrt{\frac{1}{1+\frac{|\n F^2|^2}{4F^2}}} d \mu_{Y_1}
\end{flalign*}
and so using Proposition \ref{betterS2} and equation (\ref{F2estimate}) then we see for $t$ large enough there exists a $C_\mathcal{B}$ such that
\[C_\mathcal{B}F^n\sqrt{\frac{1}{1+\frac{C_S}{\log(\widetilde{D}_S+F^2)}}} \leq \int_M d\mu \leq C_\mathcal{B} F^n\ \ .\]
Since $\psi = \left(\int_M d\mu \right)^{-\frac{1}{n}}$ we have the Lemma.
\end{proof}

\begin{theorem}
 Any initially spacelike solution of equation (\ref{pmcf}) with a convex cone boundary condition will converge as time tends towards infinity to some $\mb{G}_k$ in the $C^1$ norm. Furthermore, on any interior set uniformly away from the boundary there exists an increasing sequence of $t_i\rightarrow \infty$ such that $\widehat{M}_{t_i}$ converge to the solution on the interior in the $C^\infty$ topology. 
\end{theorem}
\begin{proof}
Under a dilation by a factor $\lambda$, $\mathcal{D}_\lambda $ we have $\mathcal{D}_\lambda(F^2) = \lambda^2 F^2$, $\mathcal{D}_\lambda S = \lambda S$, and so on. Hence from equation (\ref{F2estimate}), Proposition \ref{betterS2}, Corollary \ref{betterH} and the above estimates on the dilation factor then we get
\[
\begin{array}{rcccl}
 \widehat{C}_1 &\leq& \widehat{F}^2 &\leq& \widehat{C}_2\\
0 &\leq& |\nabla \widehat{F}^2 |^2 &\leq& \frac{\widehat{C}_S}{\log(\widetilde{D}_S +F^2)}\\
\widehat{C}^H_1 &\leq& \widehat{H} &\leq& \widehat{C}^H_2
\end{array}
\]
where $\widehat{C}_1, \widehat{C}_2,\widehat{C}_S,\widehat{C}^H_2>0$ and $\widehat{C}^H_1>0$ if $C^H_1>0$ is and all of these constants depend only on $n$ and $M_0$. The first of these is boundedness of the renormalised hypersurface, the second implies that we have in fact $C^1$ convergence to a hypersurface with $|\nabla F^2|=0$. Therefore these estimates imply $C_1$ convergence of $\widehat{M}$ to $Y_R$, the hyperbolic hyperplane of radius
\[R=\left( \int_\mathcal{B} d\mu_{Y_1} \right)^{-\frac{1}{n}}\]
where $\mathcal{B}$ is as in the previous Lemma.
 
By Lemmas \ref{A2interior} and \ref{nablaA2interior} we have that for $L$ as in Lemma \ref{A2interior} then for all time on $M_t\cap L$
\[
\begin{array}{rcccl}
 \max \left\{\frac{{\widehat{C}{}^H_1}}{n}, 0 \right\}^2 &\leq& |\widehat{A}|^2 &\leq& \widehat{C}_A\\
0 &\leq& |\nabla^m \widehat{A} |^2 &\leq& \widehat{C}_{A,m}
\end{array}
\]
where the constant depends on the boundary of $L$ and how far $L$ is from $\Sigma$.

Now using Arzel\'a--Ascoli repeatedly, a standard diagonal argument completes the theorem.
\end{proof}

\bibliographystyle{plain}
\bibliography{$HOME/Documents/bibblee/bib}
\end{document}